\documentclass[12pt]{amsart}
\usepackage{amssymb,amsfonts,amsmath,graphicx,lineno}
\usepackage{mathrsfs}
\usepackage{anysize}
\usepackage{url}
\usepackage[outline]{contour}
\usepackage{tikz}
\usetikzlibrary{plotmarks}
\usepackage{calc}
\usepackage{tabularx}
\usepackage{xcolor,framed}
\usepackage{graphicx}
\usepackage{latexsym, amsmath, amsfonts,amssymb,mathrsfs}
\usepackage{indentfirst}
\usetikzlibrary{arrows}
\usetikzlibrary{decorations.markings}
\usetikzlibrary{patterns}
\usepackage{bbding}
\usepackage{datetime}
\usepackage{pgfplots}

\theoremstyle{plain}
\newtheorem{theorem}{Theorem}[section]
\newtheorem{lemma}[theorem]{Lemma}
\newtheorem{corollary}[theorem]{Corollary}

\theoremstyle{definition}

\theoremstyle{remark}
\newtheorem{remark}[theorem]{Remark}

\newcommand{\R}{\mathbb{R}}

\title[Short Title]{Invariant measures for piecewise continuous maps}


\medskip
\begin{document}
 \maketitle

\centerline{\scshape Benito Pires \footnote{Partially supported by  S\~ao Paulo Research Foundation (FAPESP) grant \# 2015/20731-5.}}
\smallskip

{\footnotesize
 \centerline{Departamento de Computa\c c\~ao e Matem\'atica, Faculdade de Filosofia, Ci\^encias e Letras}
 \centerline {Universidade de S\~ao Paulo, 14040-901, Ribeir\~ao Preto - SP, Brazil}
   \centerline{benito@usp.br} }

\marginsize{2.5cm}{2.5cm}{1cm}{3cm}

  \begin{abstract}  We say that $f:[0,1]\to [0,1]$ is a {\it piecewise continuous interval map} if there exists a partition
$0=x_0<x_1<\cdots<x_{d}<x_{d+1}=1$ of $[0,1]$ such that $f\vert_{(x_{i-1},x_i)}$ is continuous and the lateral limits $w_0^+=\lim_{x\to 0^+} f(x)$, $w_{d+1}^-=\lim_{x\to 1^-} f(x)$, \mbox{$w_i^{-}=\lim_{x\to x_i^-} f(x)$} and
$w_i^{+}=\lim_{x\to x_i^+} f(x)$ exist for each $i$. We prove that every piecewise continuous interval map  without connections admits an invariant Borel probability measure. We also prove that every injective piecewise continuous interval map with no connections and no periodic orbits is topologically semi-conjugate to an interval exchange transformation.
 \end{abstract}
  
 \section{Introduction}
 
 Much information about the long-term behaviour of the iterates of an interval map is revealed by its invariant measures. Regarding piecewise continuous maps, the presence of a non-atomic invariant Borel probability measure can be used to construct topological conjugacies or semi-conjugacies with interval exchange transformations (IETs).
 
 Transfer operators have proved to be an important tool to obtain absolutely continuous invariant measures for piecewise smooth piecewise monotone interval maps (see \cite{VB,BG,LY}). The Folklore Theorem (see \cite{AF,RB}) claims that every piecewise expanding Markov map of the interval admits an ergodic invariant measure equivalent to the Lebesgue measure. In general, these types of results require that each branch of the piecewise continuous map be $C^r$-smooth ($r\ge 1$), monotone and have derivative greater than $1$.  
 
 The aim of this article is to prove the existence of invariant Borel probability measures for piecewise continuous interval maps not embraced by the transfer operator approach. In this way, our result includes 
 gap maps, piecewise contractions and generalised interval exchange maps (GIETs). No monotonicity and no smoothness assumptions, beyond the uniform continuity of each branch of the map, are assumed. Our result is the natural version of the Kryloff-Bogoliouboff Theorem (see \cite{KB}) for piecewise continuous interval maps.        

 We are also interested in constructing topological semi-conjugacy between injective piecewise continuous interval maps and interval exchange transformations, possibly with flips. In this regard, it is worth mentioning the result by J. Milnor and W. Thurston (see \cite{MT}), which states that any continuous piecewise monotone interval map of positive entropy ${\rm htop}\,(T)$ is topologically semi-conjugate to a map with piecewise constant slope equals to $\pm e^{{\rm htop}\,(T)}$. This result was generalised by L. Alsed\`a and M. Misiurewicz in \cite{AM} to piecewise continuous piecewise monotone interval maps of positive entropy. The same type of result was obtained by M. Misiurewicz and S. Roth in \cite{MR} for
  countably piecewise continuous, piecewise monotone intervals maps. The  author and A. Nogueira proved in \cite{NP} that every injective piecewise contraction is topologically conjugate to a constant slope map with slope equals to $\pm \frac12$.

The proof of the Kryloff-Bogoliouboff Theorem fails for discontinuous maps. The main difficulty, as pointed out by C. Liverani in \cite[p. 4]{CL}, is Lemma \ref{mi}, which is automatic for continuous functions. In this article, we present a  proof that overcomes such limitation.
 
 \section{Statement of the results}

 Throughout this article, assume that $f:[0,1]\to [0,1]$ is a piecewise continuous interval map. Hence,
 there exists a partition
$0=x_0<x_1<\cdots<x_{d}<x_{d+1}=1$ of $[0,1]$ such that $f\vert_{(x_{i-1},x_i)}$ is continuous and the lateral limits $w_0^+=\lim_{x\to 0^+} f(x)$, $w_{d+1}^-=\lim_{x\to 1^-} f(x)$, $w_i^{-}=\lim_{x\to x_i^-} f(x)$ and
$w_i^{+}=\lim_{x\to x_i^+} f(x)$ exist for each $i$. Let
$$D=\{x_0, \ldots, x_{d+1}\}, \quad W=\{w_0^+,w_{1}^-,w_1^+,\ldots,w_d^-,w_d^+,w_{d+1}^-\}.
$$ 
We say that $f$ has {\it no connections} if 
\begin{equation}\label{condition}
\bigcup_{i=0}^{d+1}\bigcup_{k=1}^\infty \{f^k(x_i)\}\cap D=\emptyset\quad\textrm{and}\quad
\bigcup_{w\in W}\bigcup_{k=0}^\infty \{f^k(w)\}\cap D=\emptyset.
\end{equation}
Notice that the first condition in $(\ref{condition})$ is implied by the second one when $f$ is left-continuous or right-continuous at each point $x_i\in D$.  We say that $x\in [0,1]$ is a {\it periodic point} of $f$ if there exists an integer $k\ge 1$ such that $f^k(x)=x$.

Our first result turns out to be a version of the Kryloff-Bogoliouboff Theorem \cite{KB} for piecewise continuous  maps. 

  \begin{theorem}\label{main1} Let $f:[0,1]\to [0,1]$ be a piecewise continuous  map with no connections, then $f$ admits an invariant Borel probability measure $\mu$. Moreover, if $f$ has no periodic points, then the measure $\mu$ is non-atomic.  
\end{theorem}

In the world of generalised interval exchange transformations, the hypothesis of no connections corresponds to the notion of having an $\infty$-complete path. As remarked in \cite[p. 1586]{MMY},
every GIET with such property is topologically semi-conjugate to an IET. The next result extends this claim to piecewise continuous maps. It can also be considered a generalisation of the item (a) of the Structure Theorem by Gutierrez \cite[p. 18]{CG}.
 
  \begin{corollary}\label{cor1} Let $f:[0,1]\to [0,1]$ be an injective piecewise continuous  map with no connections and no periodic points, then $f$ is topologically semi-conjugate to an interval exchange transformation, possibly with flips.
   \end{corollary} 
   
Now we present a class of piecewise continuous interval maps for which having no connections is a generic (in the measure-theoretical sense) property. We recall that an irrationality criterium for the absence of connections in IETs without flips was provided by M. Keane in \cite{MK}.

 \begin{theorem}\label{main2} Let $\phi_1,\ldots,\phi_{d+1}:[0,1]\to (0,1)$ be continuous maps and let $\Omega\subset\R^d$ be the open set $\Omega=\{(x_1,\ldots,x_{d})\in \R^d \mid 0 < x_1<\cdots < x_{d}<1\}$, then for Lebesgue almost every $(x_1,\ldots,x_{d})\in \Omega$, the piecewise continuous map $f:[0,1]\to (0,1)$ defined by $f(x)=\phi_i(x)$ if $x\in I_i$, where $I_1=[0,x_1), I_2=[x_{1},x_2),\ldots, I_d=[x_{d-1},x_d)$, $I_{d+1}=[x_{d},1]$, has no connections and hence admits an invariant Borel probability measure.
  \end{theorem}

\section{Proof of Theorem \ref{main1}}   

Henceforth, assume that the map $f$ has no connections and no periodic orbits.

\begin{lemma}\label{lemjx} Given  $x\in [0,1]$ and an integer $r\ge 1$, there exists an open subinterval $J_x$ of $[0,1]$ containing $x$ such that
\begin{equation}\label{Jx}
 \{f(y),\ldots, f^{r}(y)\}\cap J_x=\emptyset\quad\textrm{for every}\quad y\in J_x.
\end{equation}
\end{lemma}
\begin{proof} First let us prove the result for $x=x_i$, where $1\le i\le d$. Let 
$$\gamma=\bigcup_{k=1}^r \{f^{k-1}(w_i^-) ,f^k(x_i), f^{k-1}(w_i^+)\}.$$
The hypothesis of no connections implies that 
$ \gamma \cap \{ x_0,\ldots,x_{d+1}\}=\emptyset.
$
Hence, $f$ is continuous on an open neighborhood of $\gamma$. Moreover, for every $\epsilon>0$, there exist $0<\delta<\epsilon$ and an interval $J_{x_i}=(x_i-\delta,x_i+\delta)\subset [0,1]$ such that 
\begin{equation}\label{e}
 d\left( f^k(y), \gamma\right)<\epsilon\quad\textrm{for every}\quad y\in J_{x_i}\quad\textrm{and}\quad  1\le k\le r,
 \end{equation}
where $d\left( f^k(y), \gamma\right)=\min_{z\in \gamma} \vert f^k(y)-z \vert$.

Let $\epsilon=\frac12 d(x_i,\gamma)$, then $\epsilon>0$ because $x_i\not\in \gamma$. This together with $(\ref{e})$ implies that $\vert f^k(y)-x_i\vert >\epsilon >\delta$ for all $y\in J_{x_i}$ and $1\le k\le r$. Hence, (\ref{Jx}) holds for every $x=x_i\in D$.

The cases in which $x=x_0=0$ or $x=x_{d+1}=1$ follows likewise, by considering intervals of the form $J_{x_0}=[0,\delta)$ and $J_{x_{d+1}}=(1-\delta,1]$, respectively.

It remains to consider the case in which $x\not\in \{x_0,\ldots,x_{d+1}\}$. Due to the hypothesis of no connections,  there are only two possibilities: either
$ \{ f^k(x): k\ge 0\}\cap \{x_0,\ldots,x_{d+1}\}=\emptyset$ or there exist unique $k\ge 1$ and $0\le i\le d+1$
 such that $f^k(x)=x_i$. As for the first possibility, take $\gamma=\{f(x),\ldots,f^r(x)\}$, then $f$ is continuous on $\{x\}\cup \gamma$. Moreover, since $f$ has no periodic points, we have that $x\not\in \gamma$. Therefore, for every $\epsilon>0$, there exist $0<\delta<\epsilon$ and an interval $J_x=(x-\delta,x+\delta)$ such that
 $(\ref{e})$ holds for $J_x$ in the place of $J_{x_i}$. To conclude the proof, proceed as before. Concerning the second possibility, let $J_{x_i}=(x_i-\delta,x_i+\delta)$ be as in the beginning of the proof, then, as already proved,
\begin{equation}\label{Jx1}
 \{f(y),\ldots, f^{r}(y)\}\cap J_{x_i}=\emptyset\quad\textrm{for every}\quad y\in J_{x_i}.
\end{equation} 
 Moreover, $f$ is locally continuous around $\{x,f(x),\ldots,f^{k-1}(x)\}$, thus there exists an interval $J_x=(x-\eta,x+\eta)$ such that $J_x, f(J_x),\ldots, f^{k}(J_x)$ are pairwise disjoint intervals and $f^k(J_x)\subset J_{x_i}$. Now $(\ref{Jx1})$ implies that $(\ref{Jx})$ holds for every $y\in J_x$, which concludes the proof.
 \end{proof}

Let $q\in [0,1]$ be given. Since $f$ has no connections and no periodic orbits, there exists $\ell\ge 0$ such that 
$\left\{f^k(q): k\ge \ell \right\}\cap D=\emptyset$. Hereafter, set $p=f^\ell(q)$, then
\begin{equation}\label{p}
\{ p, f(p), f^2(p),\ldots \} \cap D=\emptyset.
\end{equation}
Denote by $(\mu_n)_{n=1}^\infty$ the sequence of Borel probability measures on $[0,1]$ defined by
 $$ \mu_n = \dfrac{1}{n} \sum_{k=0}^{n-1} \delta_{f^k(p)},
 $$
where $\delta_{f^k(p)}$ is the Dirac probability measure on $[0,1]$ concentrated at $f^k(p)$.

By the Banach-Alaoglu Theorem, the space of Borel probability measures on a compact metric space is compact with respect to the weak$\textsuperscript{$*$}$ topology. Hence, there exist a Borel probability measure on $[0,1]$, denoted henceforth by $\mu$, and a 
subsequence  of $\{\mu_n\}$, denoted henceforth by $\{\mu_{n_j}\}_{j=1}^\infty$,
 that converges to $\mu$ in the weak$\textsuperscript{$*$}$ topology.
 
 The next result is going to be used twice, in Lemma \ref{na} as well as in Lemma \ref{mi}.
 
\begin{lemma}\label{lem4} Let $x\in [0,1]$. Given $\epsilon>0$, there exist an open subinterval $J_x$ of $[0,1]$ containing $x$, and an integer $j_0\ge 1$ such that
\begin{equation}\label{xxx}
\mu_{n_j}(J_x)  <\epsilon \quad \textrm{for every}\quad j\ge j_0.
\end{equation}
\end{lemma}
\begin{proof} Let $r\ge 1$ be an integer so great that $\dfrac{2}{r}<\epsilon$. Since $\{n_j\}_{j=1}^\infty$ is a subsequence of $\{1,2,\ldots\}$, there exists $j_0\ge 1$ such that $n_{j}>r$ for every $j\ge j_0$. Let $J_x$ be as in the statement of Lemma \ref{lemjx}. Let $j\ge j_0$ and $\ell=\#\{0\le k\le n_j-1 \mid f^k(p)\in J_x\}$, where $\#$ denotes cardinality. By $(\ref{Jx})$, we have that $(\ell-1) r \le n_j$, thus 
$$ \mu_{n_j}(J_x)= \dfrac{1}{n_j} \sum_{k=0}^{n_j-1} \delta_{f^k(p)}(J_x)= \dfrac{\#\{0\le k\le n_j-1 \mid f^k(p)\in J_x\}}{n_j}\le \dfrac{2}{r}<\epsilon\,\,\textrm{for every}\,\, j\ge j_0.
$$
\end{proof}

\begin{lemma}\label{na} The measure $\mu$ is non-atomic.
\end{lemma}
\begin{proof} Let $x\in (0,1)$. Given $\epsilon>0$, let $J_x$ be an open subinterval of $[0,1]$ containing $x$ as in the statement of Lemma \ref{lem4}. Since the set $S=\{z\in [0,1]: \mu(\{z\})>0 \}$ is at most countable, there exists a subinterval $J_x'\subset J_x$ containing $x$ such that $\mu\left( \partial J_x'\right)=0$, where
 $\partial J_x'$ denotes the endpoints of the interval $J_x'$.
 By \cite[Theorem 6.1, p. 40]{P} and by $(\ref{xxx})$,
\begin{equation*}
\mu(\{x\})\le\mu(J_x')=\lim_{j\to\infty}\mu_{n_j}(J_x')\le \limsup_{j\to\infty} \mu_{n_j}(J_x)\le \epsilon.
\end{equation*}
The fact that $\epsilon$ is arbitrary yields $\mu(\{x\})=0$.

Now let $A_1\subset A_2\subset\cdots$ be a sequence of subsets of $[0,1]$ such that  $\bigcup_{k\ge 1} A_k=(0,1)$ and $\partial A_k\cap S=\emptyset$ for every $k\ge 1$. By $(\ref{p})$, we have that
$\mu_{n_j}(A_k)=1$ for every $j,k\ge 1$. By \cite[Theorem 6.1, p. 40]{P} once more, we have that
$$ \mu(A_k)=\lim_{j\to\infty} \mu_{n_j}(A_k)=1\quad\textrm{for every}\quad k\ge 1.
$$
In this way,
$$
\mu\left((0,1)\right)=\lim_{k\to\infty} \mu(A_k)=1,\quad\textrm{thus}\quad \mu\left(\{0\}\right)=\mu\left(\{1\}\right)=0.
$$
\end{proof}

The convergence of  $\{\mu_{n_j}\}_{j=1}^\infty$  to $\mu$ in the weak$\textsuperscript{$*$}$ topology implies that
$\lim_{j\to\infty}\int \phi \,{\rm d}\mu_{n_j}=\int \phi \,{\rm d}\mu$ for every
continuos function $\phi:[0,1]\to\R$. The next lemma extends this  claim for the piecewise continuous map $\phi=\varphi \circ f$.

 \begin{lemma}\label{mi} For every continuous function $\varphi:[0,1]\to\R$,
 $$ \lim_{j\to\infty} \int \varphi\circ f\,{\rm d}\mu_{n_j}=\int \varphi\circ f \, {\rm d}\mu.$$
 \end{lemma}
 \begin{proof} Let $\epsilon>0$ be arbitrarily small. By Lemma \ref{na}, we have that $\mu(\{x_i\})=0$ for every $1\le i \le d$. Hence, there exists an open interval $J_{x_i}'$ containing $x_i$ such that $\mu(J_{x_i}')<\epsilon$ for every $1\le i\le d$. By Lemma \ref{lem4}, there exist  an open interval
 $J_{x_i}''$ containing $x_i$,  and an integer $j_0\ge 1$ such that
 $$ \mu_{n_j}\left(J_{x_i}''\right) < \epsilon\quad\textrm{for every}\quad j\ge j_0\quad \textrm{and}\quad 1\le i\le d.
 $$ 
 Set $J_{x_i}=J_{x_i}'\cap J_{x_i}''$. The function $\varphi\circ f$ is bounded by some constant $M$ and continuous on each  interval
 $(x_{i-1},x_i)$ for every $1\le i\le d+1$. In this way, there exists a continuous function $h:[0,1]\to [-M,M]$ such that $h(x)=\varphi\circ f(x)$ for every $x\in [0,1]\setminus \bigcup_{i=1}^d J_{x_i}$. Putting it all together yields
 \begin{equation}\label{d1}
 \left| \int \varphi\circ f\,{\rm d}\mu_{n_j}- \int h\,{\rm d}\mu_{n_j} \right|\le \int \left|\varphi\circ f-h\right|\,{\rm d}\mu_{n_j}\le 2 M d \epsilon\quad\textrm{for every}\quad j\ge j_0,
 \end{equation} 
 and
 \begin{equation}\label{d2}
 \left| \int \varphi\circ f\,{\rm d}\mu- \int h\,{\rm d}\mu \right|\le 2M d \epsilon.
 \end{equation}
 Finally, since $h$ is continuous on $[0,1]$ and $\mu_{n_j}$ converges to $\mu$ in the weak\textsuperscript{*} topology, there exists $j_1\ge j_0$ such that
  \begin{equation}\label{d3}
 \left| \int  h\,{\rm d}\mu_{n_j}- \int h\,{\rm d}\mu \right|\le \epsilon\quad \textrm{for every}\quad j\ge j_1.
 \end{equation}
  
 It follows from the equations $(\ref{d1})$, $(\ref{d2})$ and $(\ref{d3})$ that 
$$ \left| \int \varphi\circ f\,{\rm d}\mu_{n_j}- \int \varphi\circ f\,{\rm d}\mu \right|\le (4 M d+1) \epsilon\quad\textrm{for every}\quad j\ge j_1,$$
which concludes the proof.
   \end{proof}
   
    \begin{lemma} [{\cite[Theorem 6.2, p. 147]{W}}]\label{lem66} Let $m_1$ and $m_2$ be two Borel probability measures on $[0,1]$. If $\displaystyle{\int \varphi \,{\rm d}{ m_1}=\int \varphi\, {\rm d}m_2}$ for every continuous function $\varphi:[0,1]\to\R$, then $m_1=m_2$.
  \end{lemma}
 
  Given a Borel probability measure $m$ on $[0,1]$ and an integer $k\ge 1$, let $m\circ f^{-k}$ denote the Borel measure defined by $(m\circ f^{-k})(B)=m\left( f^{-k}(B)\right)$ for any Borel set $B$. In particular, for $m=\delta_p$ we have that $\delta_p\circ f^{-k}=\delta_{f^k(p)}$.   
  
  \begin{lemma}[{\cite[Lemma 6.6, p. 150]{W}}]\label{lem55} Let 
  $\psi:[0,1]\to\R$ be a Borel-measurable function, $k\ge 1$ an integer, and $m$ a Borel probability measure on $[0,1]$, then
     $$
  \int \psi\circ  f^k {\rm d}m = \int \psi \,d (m\circ f^{-k})
  $$
  \end{lemma}
   
   \begin{lemma}\label{im} The measure $\mu$ is invariant by $f$.
   \end{lemma}
   \begin{proof} By Lemma $\ref{lem66}$ and Lemma $\ref{lem55}$ (taking $\psi=\varphi$, $k=1$ and $m=\mu$), it suffices to show that
   \begin{equation}\label{sts}
   \displaystyle{\int \varphi\circ f \, {\rm d}\mu=\int \varphi\,{\rm d}\mu}
   \end{equation}
    for every continuous function $\varphi: [0,1]\to\R$. By Lemma $\ref{mi}$, for every continuos function $\varphi:[0,1]\to\R$,
   \begin{equation}\label{222}
   \left| \int \varphi\circ f \,{\rm d}{\mu}-\int \varphi \,{\rm d}\mu\right|=\lim_{j\to\infty}\left| \int \varphi\circ f\,{\rm d}{\mu_{n_j}} - \int \varphi\, d {\mu_{n_j}} \right|
   \end{equation}
   By Lemma $\ref{lem55}$ once more (now taking $\psi=\varphi\circ f$ and $m=\delta_p$), we reach
   \begin{equation}\label{mm1}
   \int \varphi\circ f\,{\rm d}\mu_{n_j}=\dfrac{1}{n_j}\sum_{k=0}^{n_j-1}\int \varphi\circ f\, {\rm d}(\delta_p\circ f^{-k})=\dfrac{1}{n_j}\sum_{k=0}^{n_j-1}\int \varphi\circ f^{k+1}\, {\rm d} \delta_p
   \end{equation}
   Likewise, 
   \begin{equation}\label{mm2}
   \int \varphi \,{\rm d}\mu_{n_j}=\dfrac{1}{n_j}\sum_{k=0}^{n_j-1}\int \varphi\, {\rm d}(\delta_p\circ f^{-k})=\dfrac{1}{n_j}\sum_{k=0}^{n_j-1}\int \varphi\circ f^{k}\, {\rm d} \delta_p
   \end{equation}
   It follows from $(\ref{222})$, $(\ref{mm1})$ and $(\ref{mm2})$ that
   \begin{eqnarray*}\label{333}
   \left| \int \varphi\circ f \,{\rm d}{\mu}-\int \varphi \,{\rm d}\mu\right|&=&\lim_{j\to\infty}\left| \dfrac{1}{n_j}\int\sum_{k=0}^{n_j-1} \left(\varphi\circ f^{k+1} -\varphi\circ f^k\right)\,{\rm d} \delta_p 
    \right| \\ &=& \lim_{j\to\infty}\left| \dfrac{1}{n_j}\int \left(\varphi\circ f^{n_j} -\varphi \right)\,{\rm d} \delta_p \right|
    \\ & \le & \lim_{j\to\infty} \dfrac{2 \Vert f\Vert}{n_j}=0.
   \end{eqnarray*}
   Hence, $(\ref{sts})$ holds, which concludes the proof.
   \end{proof}
   
   \begin{remark} The proof of Theorem \ref{main1} follows  from Lemmas \ref{na} and \ref{im}.
   \end{remark}
   
   \section{Proof of the other results}
   
 \noindent {\bf Corollary \ref{cor1}}.\,{\it Let $f:[0,1]\to [0,1]$ be an injective piecewise continuous  map with no connections and no periodic orbits, then $f$ is topologically semi-conjugate to an interval exchange transformation, possibly with flips.}
    
   \begin{proof} By Theorem \ref{main1}, $f$ admits a non-atomic Borel probability measure $\mu$ invariant by $f$. Let $h:[0,1]\to [0,1]$ be defined by $h(x)=\mu\left([0,x]\right)$, then $h$ is a continuous nondecreasing surjective map. Let $1\le i\le d+1$ and $x,y\in (x_{i-1},x_i)$ be such that $h(x)=h(y)$. We claim that $h\left(f(x)\right)=h\left(f(y)\right)$.  Assume that $x\le y$ and $f(x)\le f(y)$, then,  the injectivity of $f$ together with the continuity of $f\vert_{(x_{i-1},x_i)}$ yield $[x,y]=f^{-1}\left([f(x),f(y)]\right)$. Hence, since $\mu$ is non-atomic,
  \begin{equation}\label{conj}
   \left\vert h\left(f(y)\right)-h\left(f(x)\right) \right\vert=\mu\left( [f(x),f(y)] \right)=\mu\left(f^{-1} \left([f(x),f(y)]\right) \right)=\mu\left([x,y]\right)= \vert h(y)-h(x)\vert .
 \end{equation}
  As for the other cases, proceed likewise to show that $(\ref{conj})$ still holds. Hence, the claim is true. 

Let $T:[0,1]\to [0,1]$ be defined by $T(h(x))=h \left(f(x)\right)$.  By the claim,  $T$ is well-defined.  Let $t_0,t_1,\ldots,t_{d+1}$ be defined by $t_0=0$, $t_{d+1}=1$ and $t_i=h(x_i)$ for every $1\le i\le d$. By $(\ref{conj})$, we have that for every $t,s\in (t_{i-1},t_i)$, there exists $x,y\in (x_{i-1},x_i)$ such that $t=h(x)$, $s=h(y)$ and
$$
 \vert T(t)-T(s)\vert=\vert h\left( f(x)\right)-h\left( f(y) \right) \vert=\vert h(x)-h(y)\vert= \vert t - s\vert\quad\textrm{for every}\quad t,s\in (t_{i-1},t_i).
$$
This proves that $T\vert_{(t_{i-1},t_i)}$ is an isometry, therefore $T$ is an interval exchange transformation, possibly with flips. By definition, $T\circ h=h\circ f$, thus $f$ is topologically semi-conjugate to $T$.    
\end{proof}

\noindent{\bf Theorem \ref{main2}}. {\it Let $\phi_1,\ldots,\phi_{d+1}:[0,1]\to (0,1)$ be continuous maps and let $\Omega\subset\R^d$ be the open set $\Omega=\{(x_1,\ldots,x_{d})\in \R^d \mid 0 < x_1<\cdots < x_{d}<1\}$, then for Lebesgue almost every $(x_1,\ldots,x_{d})\in \Omega$, the piecewise continuous map $f:[0,1]\to (0,1)$ defined by $f(x)=\phi_i(x)$ if $x\in I_i$, where $I_1=[0,x_1), I_2=[x_{1},x_2),\ldots, I_d=[x_{d-1},x_d)$, $I_{d+1}=[x_{d},1]$, has no connections and hence admits an invariant Borel probability measure.
  }
   \begin{proof} Denote by $Id$ the identity map on $[0,1]$. Set $\mathscr{C}_0=\{Id\}$. Let
\begin{equation*}\label{ckak}
\mathscr{C}_{k}=\{\phi_{i}\circ h \mid 1\le i\le d+1, h \in \mathscr{C}_{k-1} \},\quad k\ge 1. 
\end{equation*} 
For each $0\le i\le d+1$, $1\le j\le d$, $w_i\in \{w_i^-,w_i^+\}$ and $h\in \bigcup_{k\ge 0} \mathscr{C}_{k}$, the set $\{(x_1,\ldots,x_{d})\in\Omega\mid x_j=h(w_i)\}$ is the graph of a continuous function defined on $[0,1]$, thus it is a Lebesgue null set. This together with the fact that $x_0=0$ and $x_{d+1}=1$ do not belong to the range of any $h\in\bigcup_{k\ge 1}\mathscr{C}_k$ imply that the set of parameters $(x_1,\ldots,x_d)\in\Omega$ for which the map $f$ has connections is a Lebesgue null set, denoted by $N$. Let $(x_1,\ldots,x_d)\in \Omega\setminus N$, then either $f$ has a periodic point or $f$ has no periodic points and no connections. In the first case, $f$ has an an invariant Borel probability measure supported on its periodic orbits while in the second case, by Theorem \ref{main1}, $f$ admits an invariant non-atomic  Borel probability measure.
  \end{proof}
 
\section{Final remarks}

The existence of connections neither imply nor is implied by the existence of periodic points. In fact, let $f_1,f_2:[0,1]\to [0,1]$ be the piecewise affine maps defined by
$$ f_1(x)=\left\{
\begin{array}{lll}
\dfrac{x}{2} +\dfrac18 &\text{if} & 0\le x< \dfrac12\\[1em]
\dfrac{x}{2} +\dfrac38 &\text{if} & \dfrac12\le x\le 1\\
\end{array}\right.,\quad
f_2(x)=\left\{
\begin{array}{lll}
\dfrac{x}{2} +\dfrac14 &\text{if} & 0\le x< \dfrac12\\[1em]
\dfrac{x}{2} &\text{if} & \dfrac12\le x\le 1\\
\end{array}\right..
$$
The map $f_1$ has two periodic points and no connections. The map $f_2$ has a connection but no periodic points. Moreover, it does not admit any invariant Borel probability measure. In this way, the hypothesis of no connections in Theorem \ref{main1} cannot be completely removed.

\end{document}